\documentclass[a4paper,english,10pt]{article}
\usepackage[applemac]{inputenc}
\usepackage[francais]{babel}
\usepackage[T1]{fontenc}
\usepackage{graphics}
\usepackage{graphicx}
\usepackage{subfigure}
\usepackage{lscape}
\usepackage{amsmath}
\usepackage{latexsym}
\usepackage{marvosym}
\usepackage{amsfonts}
\usepackage{graphicx}
\usepackage{hyperref}
\usepackage{color}
\usepackage{enumerate}

\renewcommand{\refname}{Bibliography}

\usepackage{framed}
\usepackage[framed,thmmarks,thref,amsmath]{ntheorem}
\theoremstyle{plain}
\theoremheaderfont{\normalfont\bfseries}\theorembodyfont{\slshape}
\theoremsymbol{}
\theoremseparator{}
\newenvironment{proof}[1][Proof]{\begin{trivlist}
\item[\hskip \labelsep {\bfseries #1}]}{\end{trivlist}}

\newtheorem{theorem}{Theorem}

\newtheorem{lemme}[theorem]{Lemma}
\newcommand{\prob}{\mathbf{P}}
\newcommand{\esp}{\mathbf{E}}

\newcommand{\var}{\mathbf{V}\text{\normalfont ar}}
\newcommand{\cov}{\mathbf{C} \textrm{ov}}

\title{Optimal quantization applied to Sliced Inverse Regression}

\author{Romain Aza\"\i s, Anne G\'egout-Petit, J\'er\^ome Saracco}
\date{\textit{INRIA Bordeaux Sud-Ouest, CQFD team\\
Institut de Math\'ematiques de Bordeaux, UMR CNRS 5251 \\
Universit\'e de Bordeaux}\\~\\
January 10, 2011}

\begin{document}
\maketitle

\vspace{2cm}

\begin{center}\textbf{Abstract}\end{center}
In this paper we consider a semiparametric regression model involving a $d$-dimensional quantitative explanatory variable $X$ and including a dimension reduction of $X$ via an index $\beta'X$.  In this model, the main goal is to estimate the euclidean parameter $\beta$ and to predict the real response variable $Y$ conditionally to $X$. Our approach is based on sliced inverse regression (SIR) method and optimal quantization in $\mathbf{L}^p$-norm. We obtain the convergence of the proposed estimators of $\beta$ and of the conditional distribution. Simulation studies show the good numerical behavior of the proposed estimators for finite sample size.


\newpage

\section{Introduction}\label{intro}

In regression analysis, the main goal is to seek a parsimonious characterization of the conditional distribution of a response variable $Y$ given a $d$-dimensional explanatory variable $X$. 
In many statistical applications, the dimension $d$ of $X$ becomes large and therefore the statistical analysis becomes difficult. A usual approach to overcome this problem is to reduce the dimension of the explanatory part of the regression model without much loss of information on regression and without requiring a pre-specified parametric model. This has been achieved through the introduction of sufficient dimension reduction methods whose goal  is to reduce the dimension of $X$ by replacing it with a minimal set of linear combinations of $X$.

In this paper, we consider the following  semiparametric single index regression model
\begin{equation}\label{model}
 Y=f(\beta' X , \epsilon)
\end{equation}
where the real response variable $Y$ is linked, via an unknown link function $f$, to the $d$-dimensional random vector $X$ only through the unknown $d$-dimensional parameter $\beta$. The random variable $\epsilon$ is an error term independent of $X$. 
Model (\ref{model}) can also be defined as $Y\perp X | \beta'X$ where ``$\perp$'' stands for independence. This means that for the regression of $Y$ on $X$, a sufficient statistic is given by $\beta'X$.
 
Sliced inverse regression (SIR) introduced by Duan and Li (1991), principal hessian directions (see for instance Cook, 1998) or sliced average variance estimation (SAVE) developed by Cook (2000) are classical methods for identifying and estimating  the linear subspace spanned by $\beta$. Without additional assumption, only this subspace is identifiable  in the model and is called the central dimension reduction subspace or the effective dimension reduction (EDR) space according to the considered approach. To estimate this subspace, SIR uses properties of the conditional expectation of $X$ given $Y$ under mild assumptions on the distribution of $X$ while SAVE is based on properties of the conditional variance of $X$ given $Y$. In this paper, we focus on SIR approach which has become  the most standard method in this area because of its simple and useful estimation scheme. The main idea is to divide the whole space of $Y$ into slices and to consider the SIR matrix of interest defined as the covariance matrix of the conditional mean of $X$ in each slice. More precisely, let $\Sigma=\var(X)$ 	and let us assume that the range of $Y$ is sliced into $H$ non-overlapping slices $S_h$. Let $\hat{Y}=\{h: Y\in S_h\}$ for $h=1,\dots,H$ a discrete version of the continuous response variable $Y$. Note that $\hat{Y}$ can be seen as the projection of $Y$ on a rough grid.
Under two assumptions (the first one concerns the distribution of $X$ and the other one is in order to ensure that the regression model is not a known pathological one), it can be shown that the principal eigenvector of the matrix $\Sigma^{-1} \hat{\Gamma}$ is collinear to $\beta$ where the SIR kernel matrix $\hat{\Gamma}$ is given by  $\hat{\Gamma}=\sum_{h=1}^H \prob(Y\in S_h) (\esp[X|Y\in S_h]-\esp[X])(\esp[X|Y\in S_h]-\esp[X])'$  and then can be easily estimated by substituting empirical versions of all the moments for their theoretical counterparts.
 
 \smallskip
 
 In this paper,  we first propose to use optimal quantization in order to find an approximation of the SIR kernel matrix $\hat{\Gamma}$. A brief panorama of optimal quantization is described in Section~2 in order to remind the principle of optimal quantization which is a key stone of our method. 
 In section 3, we describe the estimator of the direction of $\beta$.  The basic idea is to replace $X$ by $\hat{X}^N$ its optimal quantizer in $\mathbf{L}^p$-norm taking a finite number $N$ of values. Let us denote $\hat{\Gamma}_N=\var(\esp\big[\hat{X}^N | \hat{Y}\big])$. We show in Section~3 that $ \hat{\Gamma}_N$ converges to $\hat{\Gamma}$ as $N$ goes to infinity and we control the rate of convergence. From this result, we will deduce the convergence of principal eigenvectors of the sequence $(\hat{\Gamma}_N)$ to the direction of $\beta$. 
In practice, optimal quantization is  frequently used to compute approximations of  conditional expectations, see for instance de Saporta et al. (2010a) and de Saporta et al.  (2010b). In Section~4 we propose to use optimal quantization to forecast $Y$ given $X$, that is $Y$ given $\beta'X$ under the considered model. We provide a theoretical result which specifies that the forecast error tends to zero as the numbers of quantizers tend to infinity. The corresponding method is particularly interesting since most of the papers on SIR in the literature only focus on the estimation of the direction of $\beta$ and do not consider the underlying regression model in its entirety. Few theoretical results combining SIR to estimate the central reduction space and nonparametric estimator of the link function are available in the literature, for instance Gannoun et al. (2004) is one of them. 
In Section~5, we illustrate our approach on simulation studies and we provide numerical results to illustrate the good behavior of the proposed estimator for finite sample size. 
 All the proofs of the mathematical results of convergence are deferred  in the Appendix.
 
\section{About optimal quantization}
\label{quantiz} 

Originally, the word ``quantization'' was used in signal and information theories by engineers since the fifties. Quantization was devoted to the discretization of a continuous signal by a finite number of ``quantizers''. It is very useful to optimize the position of the ``quantizers'' to have an efficient transmission of the signal. In mathematics, the problem of optimal quantization is to find the best approximation of the continuous distribution of a random variable by a discrete law with a fixed number $N$ of charged points. Firstly used for a one-dimensional signal, the method has been developed in the multi-dimensional case (see for instance Zador (1963) or Pag\`es (1998)) and extensively used as a tool to solve problems arising in numerical probability. It is also frequently used to solve problems in finance, as time optimal stopping, control or filtering (see for example Pag\`es et al. (2004a), Pag\`es et al. (2004b), Bally et al. (2005), Pag\`es and Pham (2005)). More recently de Saporta et al. (2010b) used quantization in order to develop a numerical method for  optimal stopping of Piecewise Deterministic Markovian Processes with an application to the optimization of reliability maintenance, see de Saporta et al. (2010a). 

Optimal quantization is well-adapted to the  approximation of  conditional expectation. In this paper we use it to tackle the estimation of the conditional distribution  of $Y$ given $\beta'X$ in the regression problem (\ref{model}). We will also specify how to use quantization in the estimation process of the SIR kernel matrix $\widehat{\Gamma}$. 

In the sequel of this section, let us first present the principle of the quantization method for a random vector $X$. Then we will provide a result on forecasting via quantization in nonparametric regression model. 

\subsection{Optimal quantization for a random vector $X$}

Let $X$ be a random vector from a probability space $(\Omega,\mathcal{F},\prob)$  to $\mathbf{R}^d$. We suppose that $X$ is finite in $ \mathbf{L}^p$-norm for some $p\geq 1$; that is to say$\| X  \|_p= \esp [ \vert X \vert^p]^{1/p}$ is finite (for $x \in \mathbf{R}^d$, $\vert x \vert $ denotes the Euclidean norm on $ \mathbf{R}^d$).
The purpose of quantization is to approximate the continuous distribution of $X$ by a discrete one with a finite support whose cardinality is $N$. 
Let  $\gamma_N$ be an $N$-grid of $\mathbf{R}^d$.
Let $\textrm{Proj}_{\gamma_N}(x)$ be the point of $\gamma_N$ which is the nearest one of $x$ for Euclidean norm. The quantization error with respect to $\gamma_N$ is
$$ 	Q_N^p ( \prob_X) (\gamma_N) = \| X - \textrm{Proj}_{\gamma_N}(X)\|_p^p	.	$$

Existence (but not uniqueness) of an optimal $N$-grid which minimizes $Q_N^p(\prob_X)(\cdot)$ vanishing its gradient has been shown under the following assumption about $\prob_X$: $\prob_X$ does not charge hyperplanes. 
From now on, for any random vector  $X$ in $\mathbf{L}^p$ which verifies this assumption, let us denote the projection on an optimal $N$-grid of $X$ by $\hat{X}^N$. 
Note that the vector $\hat{X}^N$ is a discrete random vector which verifies the following useful stationarity property:
\begin{equation}	\label{stationnarity}
\esp[ X | \hat{X}^N ] = \hat{X}^N.
\end{equation}

Some results about asymptotic quantization error have been obtained by  Zador (1963). The following theorem (see Corollary II.6.7 of Luschgy and Graph (2000)) is a generalization of a result due to Pierce (1970).  It gives the rate of convergence of the discrete approximation $\hat{X}^N$ to $X$ in $\mathbf{L}^p$ for great values of $N$ and will be very useful in our progress. 
\begin{theorem} \label{label-pierce}
If $\|X\|_{p+\delta}$ is finite for some $\delta>0$, then there exist real numbers $D_1,D_2 ,D_3$ such that for all $N \geq D_3$, we have
\begin{equation}\label{aaa}
\| X - \hat{X}^N \|_p^p \leq \frac{1}{N^{p/d}} \Big( D_1 \|X\|_{p+\delta}^{p+\delta} + D_2 \Big). 
\end{equation}
\end{theorem}

\subsection{Forecasting  in a nonparametric regression model}
\label{secsimple}

 Let us consider here the following nonparametric regression model:
\begin{equation}\label{pmodel}
Y = \tilde{f}(U, \epsilon),
\end{equation}
where $U:(\Omega , \mathcal{F}, \prob) \to \mathbf{R}^d$ is a random covariable, $\epsilon$ is a random term independent of $U$, and $ \tilde{f}$ is an unknown real link function.
We propose a method to forecast $Y$ given $U$ based on quantization approach.
First we quantize $U$ and $Y$.
Let  $\hat{U}^N$ and $\hat{Y}^N$ denote their optimal discrete approximations.
We denote by $\hat{P}$ the transition matrix from $\hat{U}^N$ to $\hat{Y}^N$, that is to say, if $\gamma_{N}$ and $\delta_N$ are optimal $N$-grids for quantization of $U$ and $Y$,

$$\forall \hat{u} \in \gamma_N,~\forall \hat{y} \in \delta_N,~\hat{P}(\hat{u},\hat{y}) = \prob(\hat{Y}^N = \hat{y} | \hat{U}^N = \hat{u}).$$

Consider the discrete random variable $\hat{Y}^c$ such that $\big( \hat{U}^N , \hat{Y}^c \big)$ is a stopped Markov chain with transition matrix $\hat{P}$. We propose $\hat{Y}^c$ as predictor for $Y$ given $U$. Actually, for a fixed $u$, we forecast $Y=\tilde{f}(u , \epsilon)$ by the conditional law of $\hat{Y}^N$ given $\hat{U}^N = \textrm{Proj}_{\gamma_N}(u)$.  So in all the results concerning the forecast of $Y$, it will be equivalent to put $\hat{Y}^c $ or $\hat{Y}^N$. We specify in Theorem~\ref{label-pred1}, how for a fixed $u$, the discrete law of probability given by $\hat{P}(\textrm{Proj}_{\gamma_N}(u),\hat{y})$ on the $N$ points $\hat{y}$ of the grid $\delta_N$, is a good approximation of the distribution of $Y$ given $U=u$. 
Knowing this distribution, we propose to use $\esp [ \hat{Y}^c | U=u]$ as a predictor of $Y$. Note that it is also easy to get conditional quantiles and a forecast interval for $Y$ with a given confidence level.

\medskip

In order to get the theoretical result, let us introduce the following assumptions:
\begin{displaymath}
\begin{array}{ll}
(\mathcal{A}_1) & \exists~ p\geq 1 ~ \textrm{s.t.} ~ U,Y \in \mathbf{L}^p,\\
 (\mathcal{A}_1')	&	\exists~ \delta >0 ~\textrm{s.t.}~ U,Y \in \mathbf{L}^{p+\delta},\\
(\mathcal{A}_2) & \exists~ [\tilde{f}]_{Lip}>0~\textrm{s.t.$~\forall u, v \in \mathbf{R}^d$,}~~ \| \tilde{f}(u,\epsilon) - \tilde{f}(v,\epsilon) 
		\|_p \leq [\tilde{f}]_{Lip} |u-v|, \\
(\mathcal{A}_3) & \textrm{The distributions of $U$ and $Y$ do not charge hyperplanes}.
\end{array}
\end{displaymath}

The following theorem gives the convergence of the conditional distribution of $\hat{Y}^N$ given $ \hat{U}^N$ to the conditional law of $Y$ given $U$ in the regression model ($\ref{pmodel})$. 
Actually, let $\phi$ be a Lipschitz function; we show that  the $\mathbf{L}^1$-norm of $\esp[\phi(Y) | U] - \esp\big[ \phi( \hat{Y}^N) | \hat{U}^N] $ is bounded by a quantity involving the $\mathbf{L}^p$-errors of quantization of the variables $Y$ and $U$. Thanks to Theorem~\ref{label-pierce}, the forecast error decreases to $0$ as $N$ approaches to infinity with rate $N^{-1/d}$.

\begin{theorem}
\label{label-pred1}
For all Lipschitz function $\phi$ with Lipschitz constant $[\phi]_{Lip}$, under assumptions $(\mathcal{A}_{1\to 3})$, we have
$$\Big\| \esp[\phi(Y) | U] - \esp\big[ \phi( \hat{Y}^N) | \hat{U}^N \big] \Big\|_1 \leq 2 [\phi]_{Lip} [\tilde{f}]_{Lip} \|U-\hat{U}^N\|_p + [\phi]_{Lip} \|Y-\hat{Y}^N\|_p.$$
Moreover, if assumption $(\mathcal{A}_1)$ is replaced by assumption $(\mathcal{A}_1')$, the rate of convergence is given by 
$$\Big\| \esp[\phi(Y) | U] - \esp\big[ \phi( \hat{Y}^N) | \hat{U}^N \big] \Big\|_1  =  \mathcal{O}\Big( \frac{1}{N^{1/d}} \Big).$$
\end{theorem}
The proof of this theorem is deferred in the Appendix A.1.

\medskip

In the next two sections, we will first estimate the linear subspace spanned by $\beta$ in model (\ref{model}) using optimal quantization. Then, we combine the proposed estimator of the EDR direction with the previous forecasting approach based on optimal quantization in order to predict $Y$ given $X$ in regression model (\ref{model}).

\section{Estimation method of the direction of $\beta$}
\label{sir}

Now, consider the semiparametric regression model (\ref{model}). Let us first remark that, since $\beta$ and $f$ are simultaneously unknown in this model, we can only identify the linear subspace spanned by $\beta$. 
Let us also recall that $\hat{\Gamma}_N$ denotes the covariance matrix of $\esp[\hat{X}^N | \hat{Y} ]$ where $\hat{Y}=\textrm{Proj}_\gamma ( Y)$ is the projection on a (non necessary optimal) grid $\gamma$ of $\mathbf{R}$.
Let $\tilde{\beta}_N$ be a principal eigenvector of the matrix $\Sigma^{-1} \hat{\Gamma}_N$.
The next result exhibited in Theorem~\ref{th-cosca} says that, for a large $N$, the direction of $\tilde{\beta}_N$ is a good approximation of the one of $\beta$. Actually,  we also give in Theorem~\ref{label-estimation1} a stronger result: there exists a sequence $(\beta_N)$ of principal eigenvectors of  $\Sigma^{-1} \hat{\Gamma}_N$ which converges to $\beta$ when the number $N$ of quantizers of $X$ goes to infinity. This result is due to the fact that the matrices $\hat{\Gamma}$ and $\hat{\Gamma}_N$ are very close for great $N$ (see Lemma  \ref{lemma1} given in Appendix A.2.). 

We need the following additional assumptions which are usually assumed in SIR framework: 
\begin{displaymath}
\begin{array}{ll}
(\mathcal{A}_5)		&	\exists \hat{y} \in \gamma,~ \esp[ (X-\mu)'\beta | \hat{Y}=\hat{y}] \neq 0,	\\
(\mathcal{A}_6)		&	\textrm{$X$ has an elliptically symmetric distribution.}
\end{array}
\end{displaymath}
The next two assumptions enable optimal quantization of $X$:
\begin{displaymath}
\begin{array}{ll}
(\mathcal{A}_7) 	&  \exists~p \geq 1~\textrm{s.t.}~X \in \mathbf{L}^p \cap \mathbf{L}^q ~\textrm{with $\frac{1}{p}+\frac{1}{q}=1$},\\
(\mathcal{A}_{8}) 	& \textrm{The distribution of $X$ does not charge hyperplanes,}\\
(\mathcal{A}_{9})	&	\exists~ \delta>0 ~\textrm{s.t.}~ X \in \mathbf{L}^{p+\delta}.
\end{array}
\end{displaymath}

\noindent

The next result gives the convergence of the direction of $(\tilde{\beta}_N)$,  for any sequence $(\tilde{\beta}_N)$ of principal eigenvectors of the sequence $(\Sigma^{-1} \hat{\Gamma}_N)$ to the direction of $\beta$  as the number of quantizer $N$ tends to infinity. For this, we need to define $\cos^2 (\tilde{\beta}_N , \beta)=\displaystyle \frac{ (\tilde{\beta}'_N \beta)^2 }{{ ({\tilde{\beta}}'_N {\tilde{\beta}}_N}) \times({ \beta' \beta})}$.

\begin{theorem}\label{th-cosca}
Under $(\mathcal{A}_{5\to 8})$,  for any sequence $(\tilde{\beta}_N)$ of principal eigenvectors of the sequence $(\Sigma^{-1} \hat{\Gamma}_N)$, we have
\[ \cos^2(\tilde{\beta}_N , \beta) \to 1 ~~\textrm{as}~N \to \infty.
 \]
\end{theorem}
The proof of this theorem is deferred in the Appendix A.3.
For the next theorem, we recall that for $x \in \mathbf{R}^d$, $\vert x \vert $ denotes the Euclidean norm on $ \mathbf{R}^d$.
\begin{theorem}\label{label-estimation1}
Under $(\mathcal{A}_{5\to 8})$, there exists a sequence of principal eigenvectors of $\Sigma^{-1}\hat{\Gamma}_N$ denoted by $(\beta_N)$ which converges to $\beta$. Indeed , there exist real constants  $C_1 , C_2 >0$ such that
\[ \forall N \geq C_2,~~ \big|\beta_N - \beta \big| \leq \frac{2 d }{C_1} \big\|\Sigma^{-1} \big\|_{\infty} \big\| X - \hat{X}^N \big\|_p \big\| X \big\|_q. \]
 Moreover, under  $(\mathcal{A}_{9})$, we control the rate of convergence by 
 \[	 |\beta_N - \beta| = \mathcal{O}\Big( \frac{1}{N^{1/d}}\Big).
\]
\end{theorem}
The proof of this theorem is deferred in the Appendix A.2.

\section{Forecasting method}
\label{forecast}

Now we mix the previous estimations to tackle the forecast of $Y$ in the semiparametric regression model (\ref{model}). For a  given $\beta_N$  defined in Section~\ref{sir},  the model 
\[
 Y=f(\beta_N' X , \epsilon)
\]
is a good approximation of the initial one. So, the method to forecast $Y$ in (\ref{pmodel}) when $U=\beta_N' X$ will give a good forecast of $Y$ in  model (\ref{model}). For the asymptotics, we have to consider two parameters: the number $N$ of quantizers of $X$ which indexes $\beta_N$ (see Section \ref{sir}) and the number $m$ of quantizers of $U=\beta_N' X $ and $Y$ (see Section \ref{secsimple}). We will show the forecast error tends to $0$ as  $m$ and $N$ simultaneously tend to infinity.

 Using Theorem~\ref{label-estimation1}, there exist some $\Delta_0 , C_0 >0$ such that for all $N \geq C_0$
\begin{equation} \label{etoile}
|\beta_N - \beta| \leq \Delta_0 \| X - \hat{X}^N \|_p.
\end{equation}
We need some additional assumptions to get our asymptotic result. Let us assume that $f$ is Lipschitz and optimal quantization in $\mathbf{L}^p$-norm for $Y$ is possible:
\begin{displaymath}
\begin{array}{ll}
(\mathcal{A}_{10}) & \exists~ [f]_{Lip}>0~ \textrm{s.t.$~\forall u, v \in \mathbf{R}$,}~~ \| f(u,\epsilon) - f(v,\epsilon)\|_p \leq [f]_{Lip} |u-v|.\\
(\mathcal{A}_{11})		&	Y \in \mathbf{L}^{p+\delta}. \\
(\mathcal{A}_{12}) 		& 	\textrm{The distribution of $Y$ does not charge hyperplanes.}
\end{array}
\end{displaymath}
We forecast $Y$ given $X$ by the random variable $\hat{Y}^c$ such that $\big( \widehat{ \beta'_N X}^m  , \hat{Y}^c \big)$ is a stopped Markov chain with the same transition matrix as $\big( \widehat{ \beta'_N X}^m  , \hat{Y}^m \big)$ where $\widehat{ \beta'_N X}^m$ and $\hat{Y}^m$ are optimal (in $\mathbf{L}^p$-norm) discrete approximations of $\beta_N' X$ and $Y$ with $m$ quantizers.
\begin{theorem} \label{label-forecast1}
Under $(\mathcal{A}_{5 \to {12}})$, for all Lipschitz function $\phi$, there exist three real numbers $A_1,A_2, A_3$, a sequence $(g_N)$ which admits a strictly positive limit and two integers $\overline{m}$ and $\overline{N}$ such that for all $m \geq \overline{m}$ and $N \geq \overline{N}$, we have
$$
	\Big\| \esp\big[ \phi(Y) | \beta'X \big] - \esp\Big[ \phi  ( \hat{Y}^m  ) | \widehat{ \beta'_N X}^m \Big] \Big\|_1     \leq \frac{A_1}{N^{1/d}} + \frac{A_2}{m} g_N+ \frac{A_3}{m}.
$$
\end{theorem}
The proof of this theorem is deferred in the Appendix A.4.
This theorem yields that the forecast error decreases to $0$ as $N$ and $m$ go to infinity.

\section{Simulation study}
\label{simu}

The aim of this simulation study is twofold. 
First, we are only interested in the estimation of the regression slope parameter $\beta$, next we focus on  forecasting $Y$ given $X=x$ in the semiparametric regression model (\ref{model}). 
In the first part, since only the direction of $\beta$ is identifiable in (\ref{model}),  we evaluate the quality of our estimator with the square cosine of the angle between the true direction $\beta$ and its estimates. The closer this square cosine is to one, the better is the estimation. 
For forecasting part of this study, we will compare the true conditional expectation and the true conditional variance of $Y$ given $X$ in model (\ref{model}) to their estimations via sliced inverse regression and quantization. For this we need to ensure that $\beta$ and its estimates have the same norm and sign before quantizing the estimated index.

\medskip

In this simulation study, we consider the following three models:
$$\begin{array}{llll}
({\mathcal{M}}_1):~~&Y &=&(\beta'X)^3 + \epsilon, \\
({\mathcal{M}}_2):~~&Y &=& (\beta'X)^3 + \beta'X \epsilon,  \\
({\mathcal{M}}_3):~~&Y &=& (\beta'X)^2 \exp\Big( \frac{\beta'X}{\theta} \Big) + \epsilon, 
\end{array}$$
where $X$ follows a $d$-dimensional normal distribution ${\mathcal{N}}(0,I_d)$ and $\epsilon$ is standard normally distributed. The error term $\epsilon$ is independent of the covariable $X$.  In the first (resp. second) part of the simulation study, the dimension of $X$ is $d=10$ (resp. $d=4$) with $\beta = (1 , -1 , 0, \dots , 0)'$. The first and third models are homoscedastic while the second one is heteroscedastic. We have introduced the parameter $\theta$ in model $({\mathcal{M}}_3)$ in order to point out the efficiency of our method even if the model is symmetric dependent. Indeed, $\theta$ controls the symmetric dependency between the index $\beta'X$ and the variable of interest $Y$. When $\theta=1$, there is none symmetric dependency and SIR works well. Symmetric dependency appears as $\theta$ increases, it is moderate when $\theta=5$ and strong when $\theta=10$. In such cases, classical SIR fails to estimate the direction of $\beta$, this is a known pathological situation for SIR. Our proposed approach will appear  more robust in presence of symmetric dependent  model.

\subsection{Estimation of the subspace spanned by $\beta$}

Assume we have a sample $\{(X_i , Y_i),~ i=1,\dots,n\}$ of random variables generated by one of the three previous regression models. 
First the covariance matrix $\Sigma$ of $X$ is estimated by the empirical one of the $X_i$'s and is denoted by $S$. 
Then $X$ and $Y$ are quantized in $\mathbf{L}^2$-norm by the usual algorithm given for example by Pag\`es and Printems (2003). 
We get the corresponding two optimal approximations $\hat{X}^N$ and $\hat{Y}^m$. Finally the covariance matrix $\hat{\Gamma}_N=\var(\mathbf{E}[\hat{X}^N | \hat{Y}^m ])$ is calculated.

We use two sample sizes $n=300$ and $n=1000$.
Since the dimension of $X$ is $d=10$, the sample size $n$ is rather small for the stochastic quantization algorithm in order to get optimal quantization grids; thus we may get some not so optimal quantization grids. 
To overcome this failure, we will use the idea of the pooled slicing approach introduced by Saracco (2001).  Let us quantize $B$ times the variables $X$ and $Y$  with the same sample. So we get $B$ estimations $T_1, \dots, T_B$ of $\hat{\Gamma}_N$ and we work with the mean
$$\overline{T} = \frac{1}{B} \sum_{b=1}^B T_b.$$
Note that since the principal eigenvector of each $\Sigma^{-1}T_b$ is collinear to $\beta$, the principal eigenvector of $\Sigma^{-1}\overline{T}$ is also collinear to $\beta$.
Let us now denote by $\hat{\beta}_N$ a principal eigenvector of $S^{-1}\overline{T}$. We will use this vector $\hat{\beta}_N$ as our estimate of the direction of $\beta$.

In the displayed results, we take $B=5$ quantization grids. The numbers of quantizers for $Y$ and $X$ are respectively  $m=5$ and $N\in\{20,30,50,100,200\}$. 

We generate 100 samples for each model and each sample size.
For each simulated sample, we calculate the estimate $\hat{\beta}_N$ and the corresponding quality measure $\cos^2(\beta, \hat{\beta}_N)$.
 
We represent in Figure \ref{figure1} (resp. Figure \ref{figure2}) the boxplots of the squared cosines according to the number $N$ of quantizers for $X$ for the various models when $n=300$ (resp. $n=1000$). We also compare our estimation method with the usual SIR method when the number $H$ of slices is equal to 5 (that is equal to the number $m$ of quantizers we used for $Y$ in our estimation process).
Clearly we observe that when the number $N$ of quantizers increases the quality of the estimator increases too for all the models. Moreover not surprisingly when the sample size $n$ becomes bigger, the squared cosines increase toward one. One can see that the classical SIR approach works better than our proposed estimation method for models $({\mathcal{M}}_1)$,  $({\mathcal{M}}_2)$ and  $({\mathcal{M}}_3)$ with $\theta=1$, that is when the underlying model favors SIR. 
We can however remark that the numerical performances of our estimator are close (resp. very close) to those obtained with classical SIR when $N=200$ and $n=300$ (resp. $n=1000$).
When a symmetric dependency appears in the regression model, our approach outperforms classical SIR: this point  is more particularly obvious when $\theta=5$ or 10 for model $({\mathcal{M}}_3)$ when $n=1000$.
Since in practice with a real dataset it is not possible to determine if the underlying regression model is symmetric dependent or not, we can expect that the use of our estimation method will provide a more robust estimation of the direction of $\beta$ than SIR.

\subsection{Forecasting}

In this part, we consider samples of size $n = 10 000$ generated from the previous semiparametric regression models where the dimension of the covariable $X$ is equal to $d=4$. 

For each simulated sample, we first estimate $\beta$ by the procedure described in the previous subsection. 
However we use only one quantization ($B=1$) for the two variables $X$ and $Y$  because we work here with enough data for the quantization algorithm in order to get an optimal quantization grid.
To do this, we use $N=200$ quantizers for $X$ and $m=5$ quantizers for $Y$. Here we need to get an estimator $\beta_N$ of $\beta$ and not only of its direction. Since we know the true slope parameter $\beta$ in simulation,  we can assume that the sign and the norm of $\beta$ are known. 

Finally, in order to estimate the conditional distribution of $Y$ given $X$, we quantize $\beta_N' X$  with $m=100$ quantizers and we also use  $m=100$ quantizers for the quantization of $Y$ in this step. Then we get the law of $\hat{Y}^m$ given $\widehat{\beta_N'X}^m$. 
Now we can estimate the conditional expectation and the conditional variance of $Y$ given $X$ by the ones of $\hat{Y}^m$ given $\widehat{\beta_N'X}^m$.

In Table \ref{tab:x1} (resp. \ref{tab:x2}), we present some estimation results of  the conditional expectation and the conditional variance of $Y$ given $X= (0.5 , -0.5,1,0)'$ (resp. given $X=(-1/3 , 0.5,1,1)'$).
For each model, we compare our estimations with the true values of the conditional expectation and the conditional variance, and we evaluate the corresponding relative error. One can see that the estimated values obtained with our proposed method are very close to the true ones for the conditional expectation as well as for the conditional variance whatever the model is. The relative errors (in absolute value) are lower than 16\% for most of the models.

Finally, we generate ten values $x_j$ from a uniform law on $[-2,2]^4$. For each value $x_j$  and for each model,  we estimate 
 the conditional expectation and the conditional variance of $Y$ given $X=x_j$. Figure~\ref{figure3} gives the two boxplot of the relative error of these conditional moments for model $({\mathcal{M}}_3)$ with $\theta=5$. One can observe that our estimation procedure provides reasonable values for these relative errors. We obtain very similar results (not given here) for the other models.

In this subsection, we only focus on the first two conditional moments of $Y$ given $X$. However, with the proposed approach, it is straightforward to make forecasting by using the conditional median for instance. In the same spirit, it is possible to estimate conditional quantiles (5\% and 95\% conditional quantiles) in order to provide the 90\% predictive interval.

\section{Concluding remarks}
\label{conclu}
We have presented a method using the probabilistic tool of quantization in order to tackle both the problem of the estimation of the EDR direction and the forecasting of $Y$ given $X$ by its conditional law in the semiparametric regression model (\ref{model}). We proved the convergence of our estimators and gave their rate of convergence. To our knowledge, using of optimal quantization in nonparametric or semiparametric statistics is new and forecasting the variable of interest by its conditional distribution with this kind of appraoch is original. The simulation studies give good results for large samples ($n=1000$) and in this case, the performance is comparable to the one of the SIR method. It is even better when a symmetric dependency appears in the regression model. In practice with a real dataset, we do not know if the underlying regression model is symmetric dependent or not, so we can expect that using of our estimation method will provide a more robust estimation of the direction of $\beta$ than SIR.

\paragraph{Acknowledgement} Professor Fran\c cois Dufour gave us the main idea of this paper that is to use together quantization and SIR methods for the problem of dimension reduction and forecasting in model (\ref{model}). We are most grateful to him for this  suggestion.

\begin{landscape}
\begin{figure}[!htbp]
\centering
\begin{tabular}{ccc}
   \includegraphics[width=7cm, height=6cm]{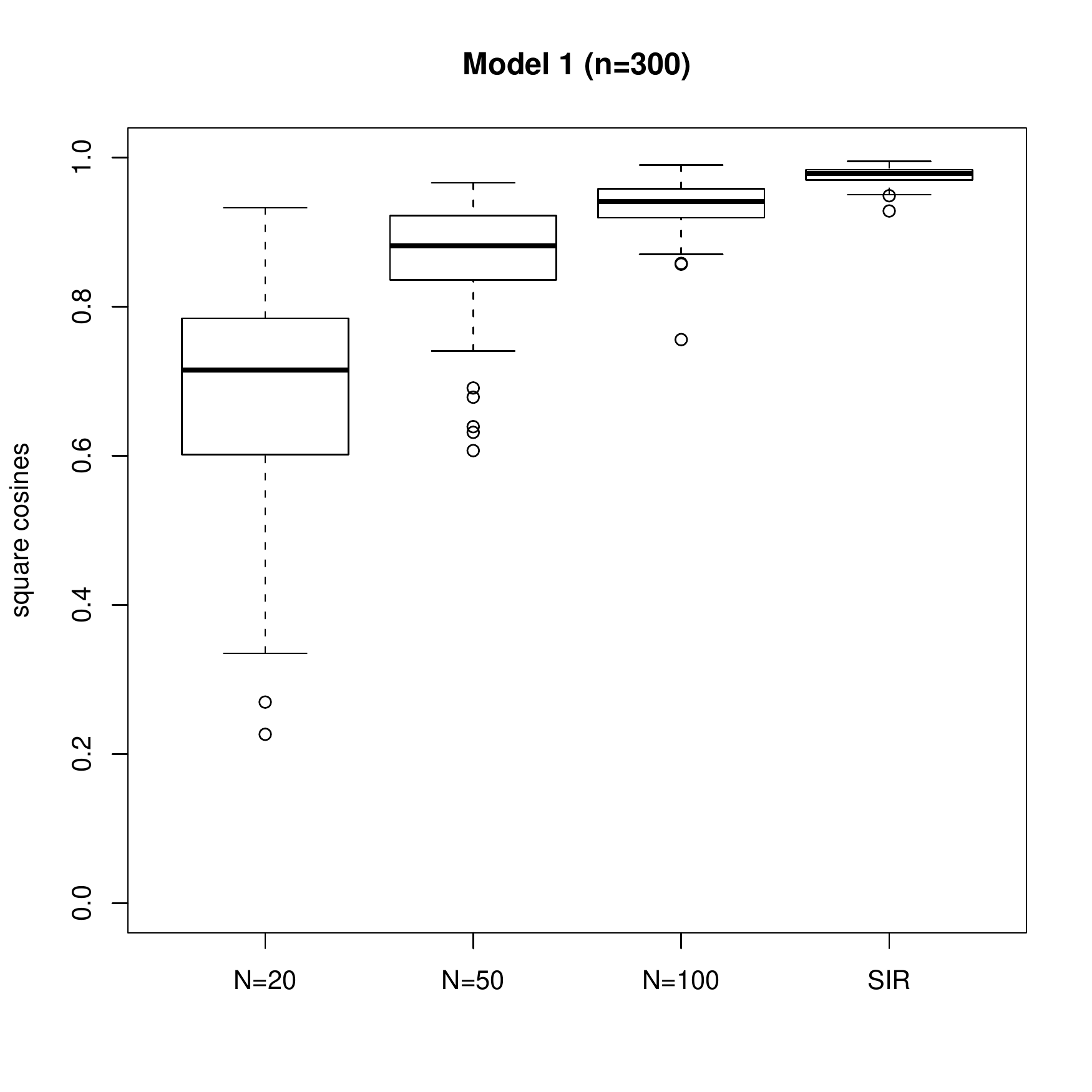} &
   \includegraphics[width=7cm, height=6cm]{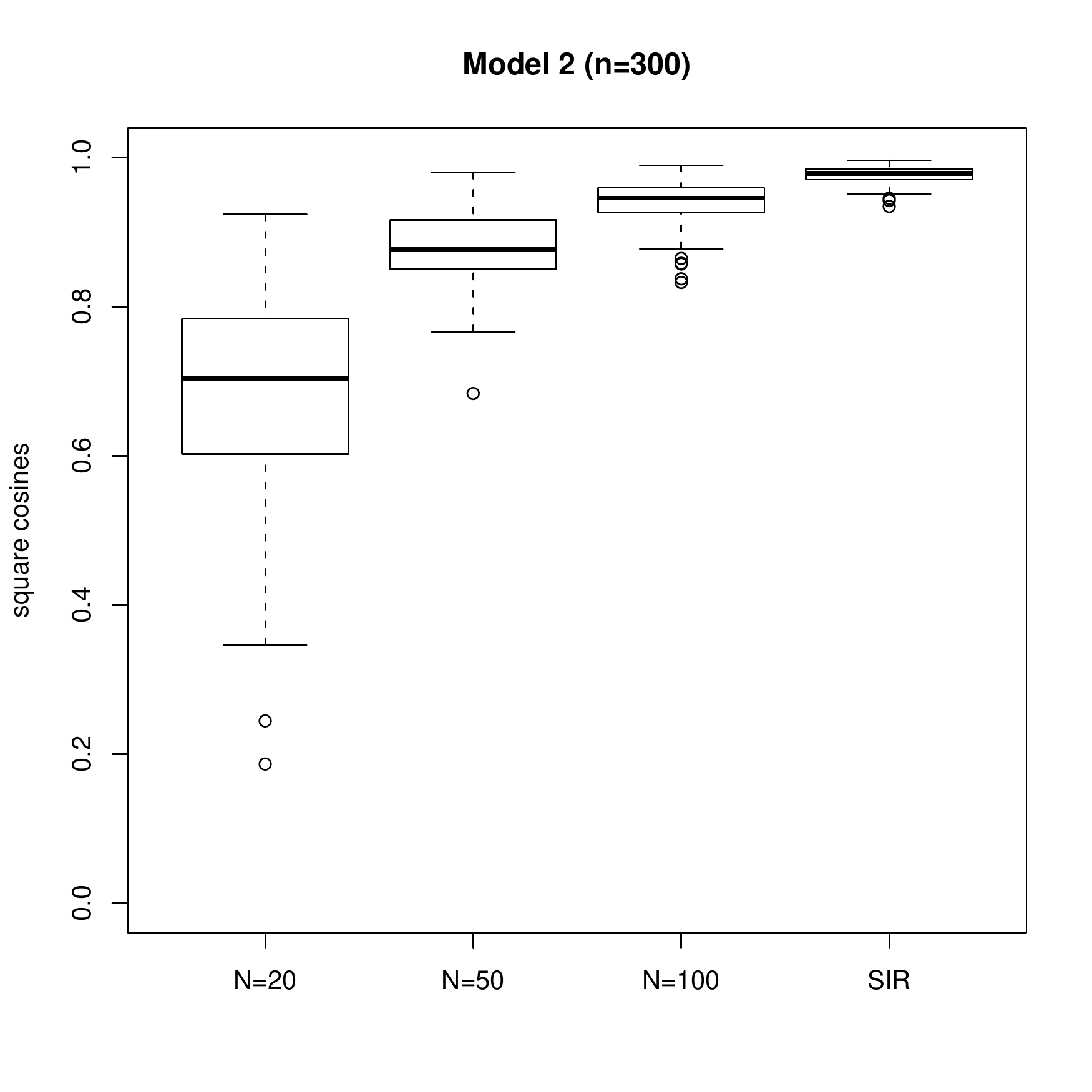} &
   \includegraphics[width=7cm, height=6cm]{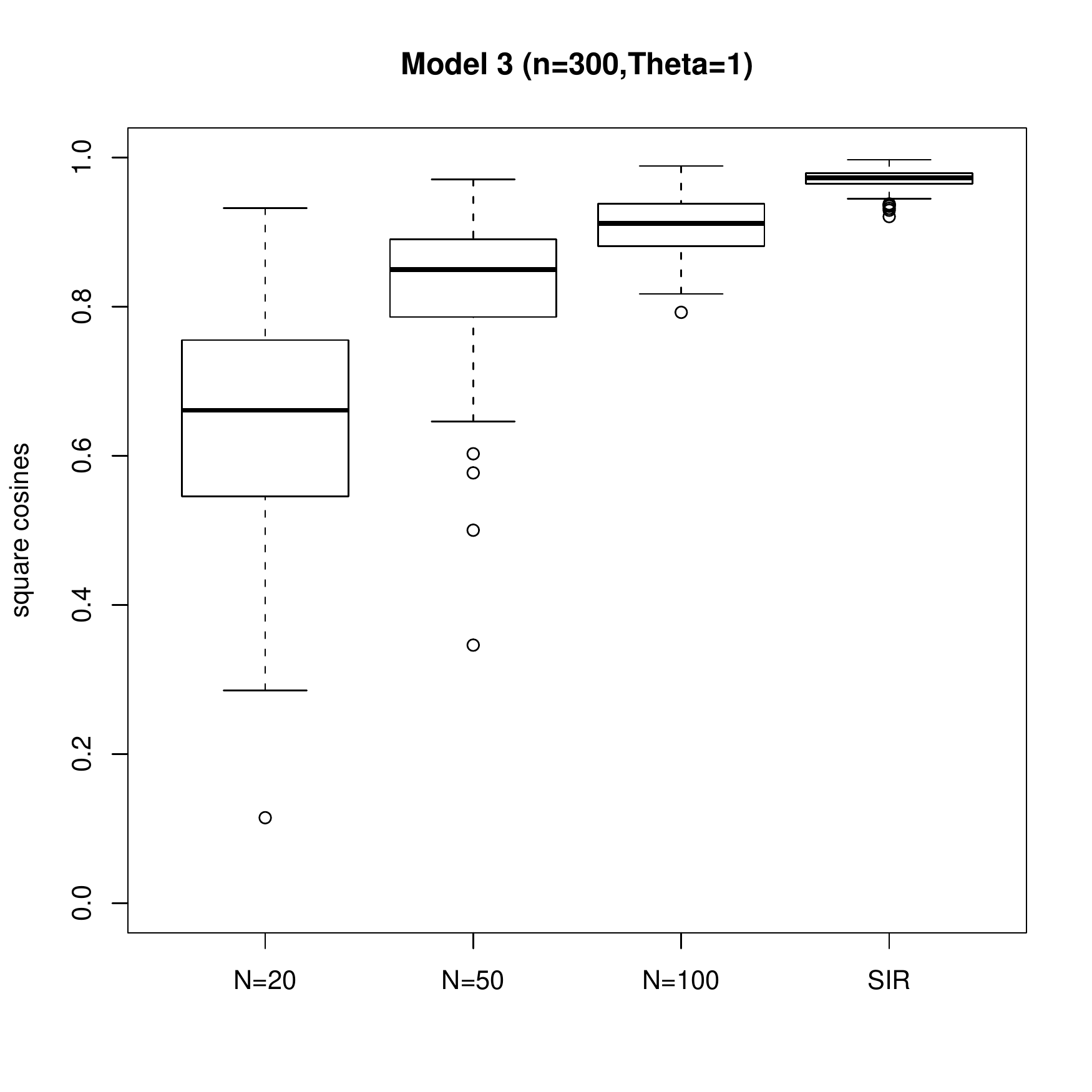} \\
   \includegraphics[width=7cm, height=6cm]{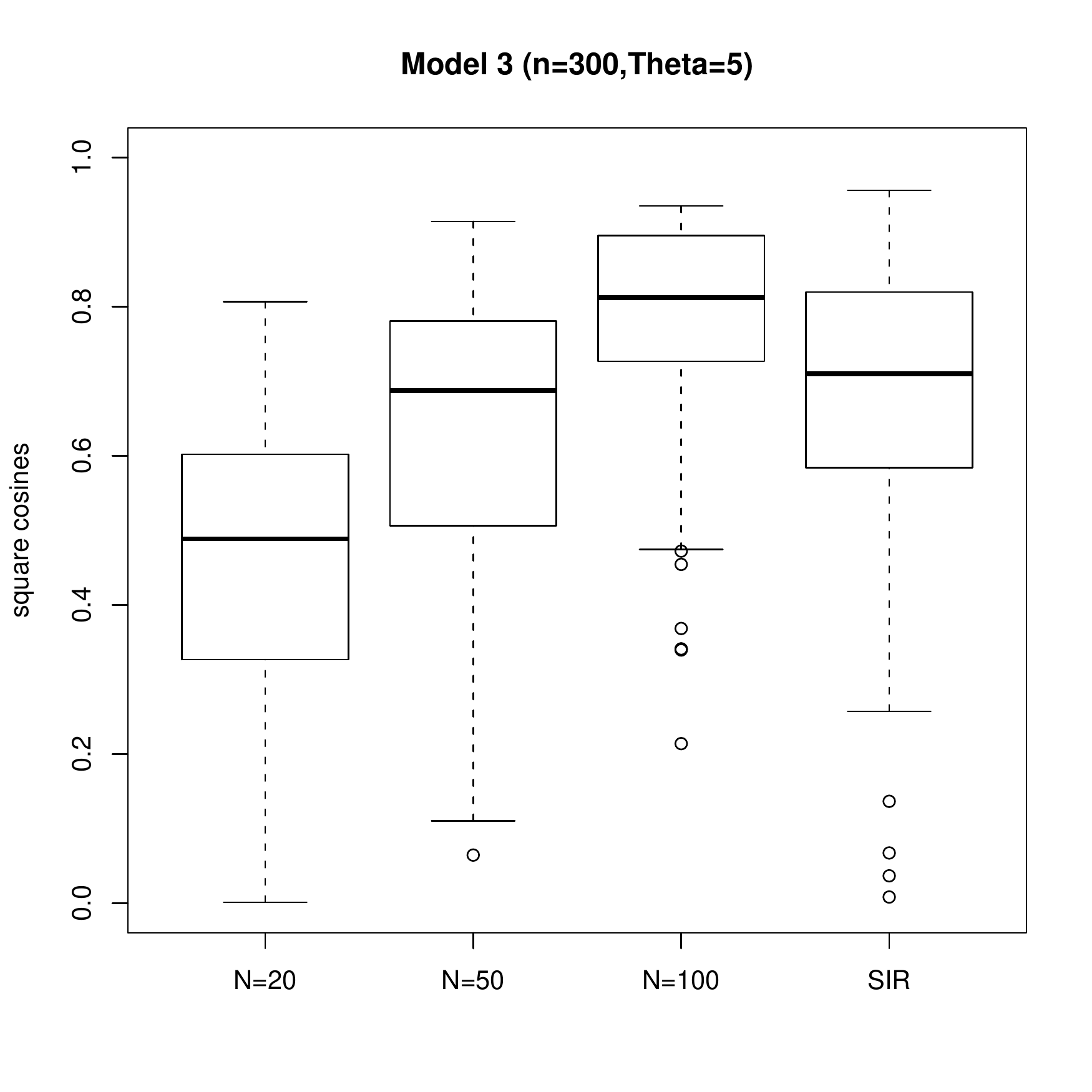} &
   \includegraphics[width=7cm, height=6cm]{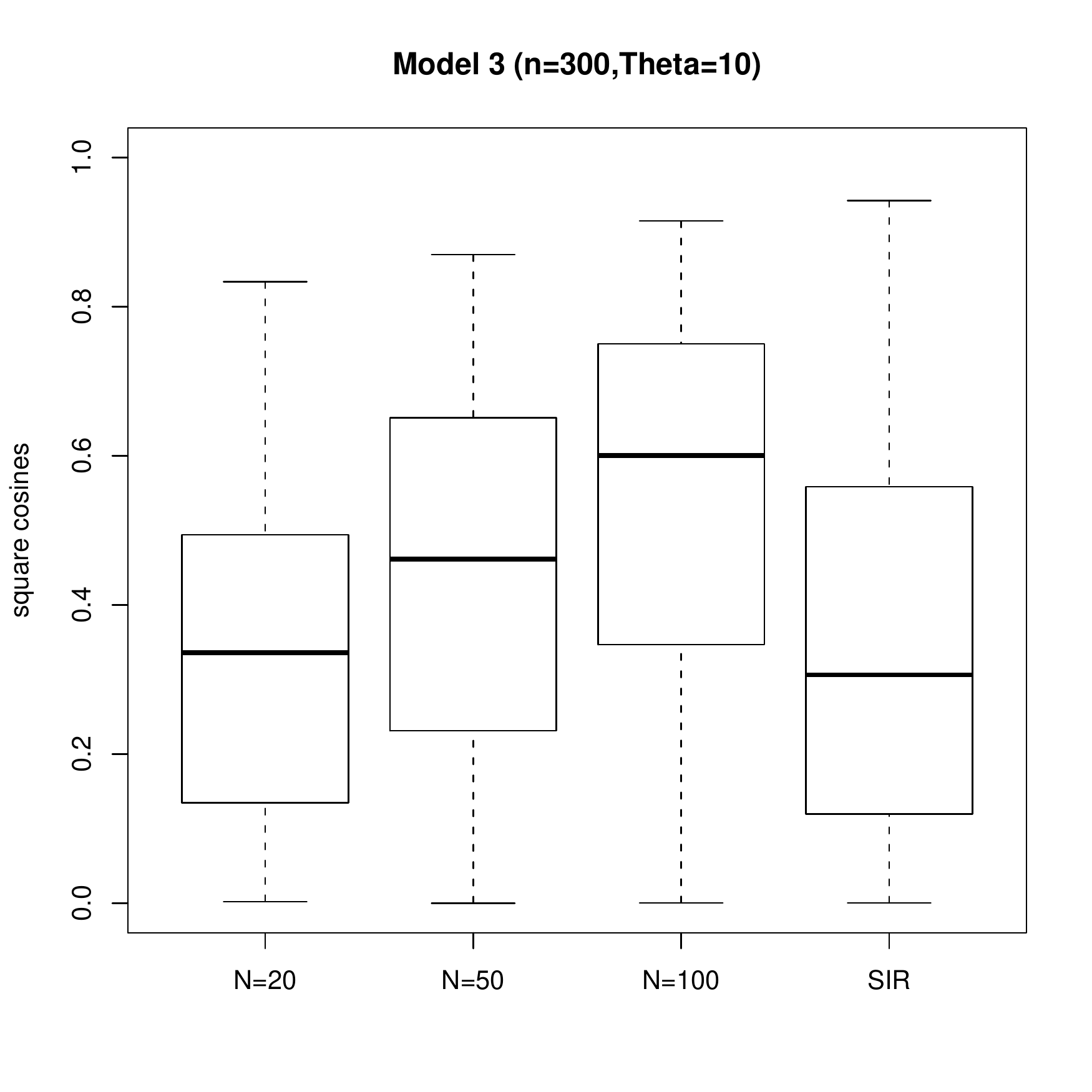}
\end{tabular}
\caption{Boxplots of the square cosines in the estimation of the regression parameter $\beta$ with $n=300$} \label{figure1} 
\end{figure}
\end{landscape}

\begin{landscape}
\begin{figure}[!htbp]\centering
\begin{tabular}{ccc}
   \includegraphics[width=7cm, height=6cm]{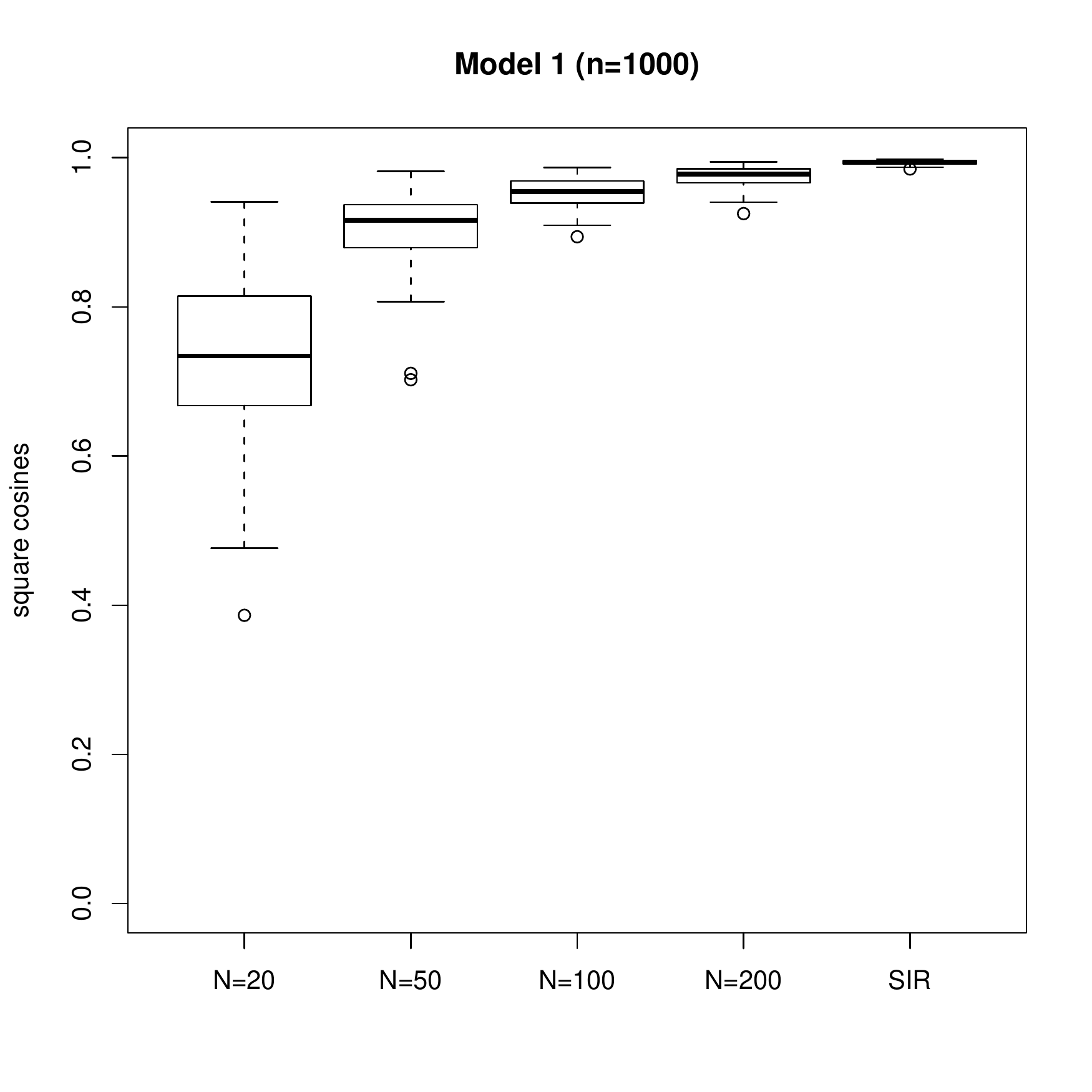} &
   \includegraphics[width=7cm, height=6cm]{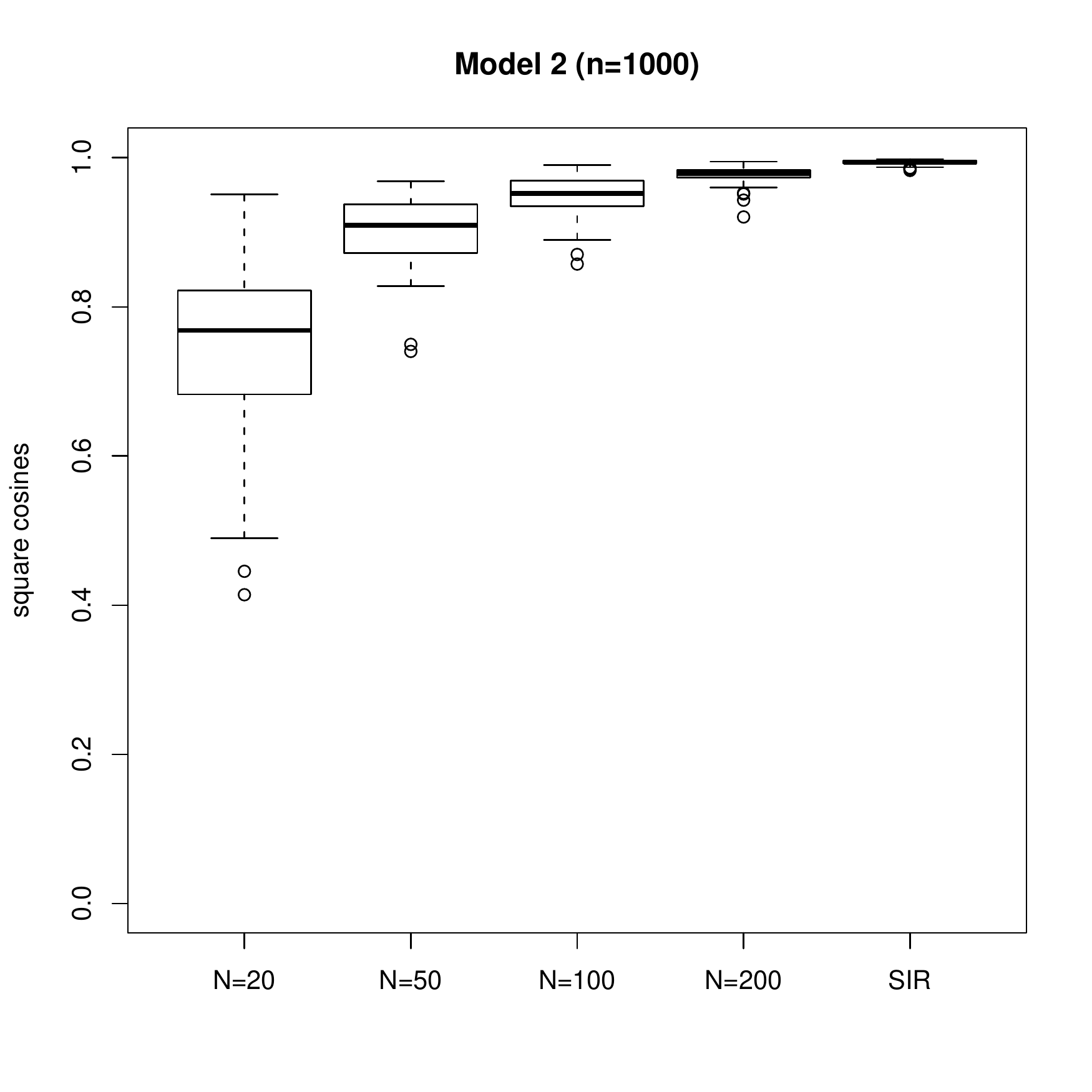} &
   \includegraphics[width=7cm, height=6cm]{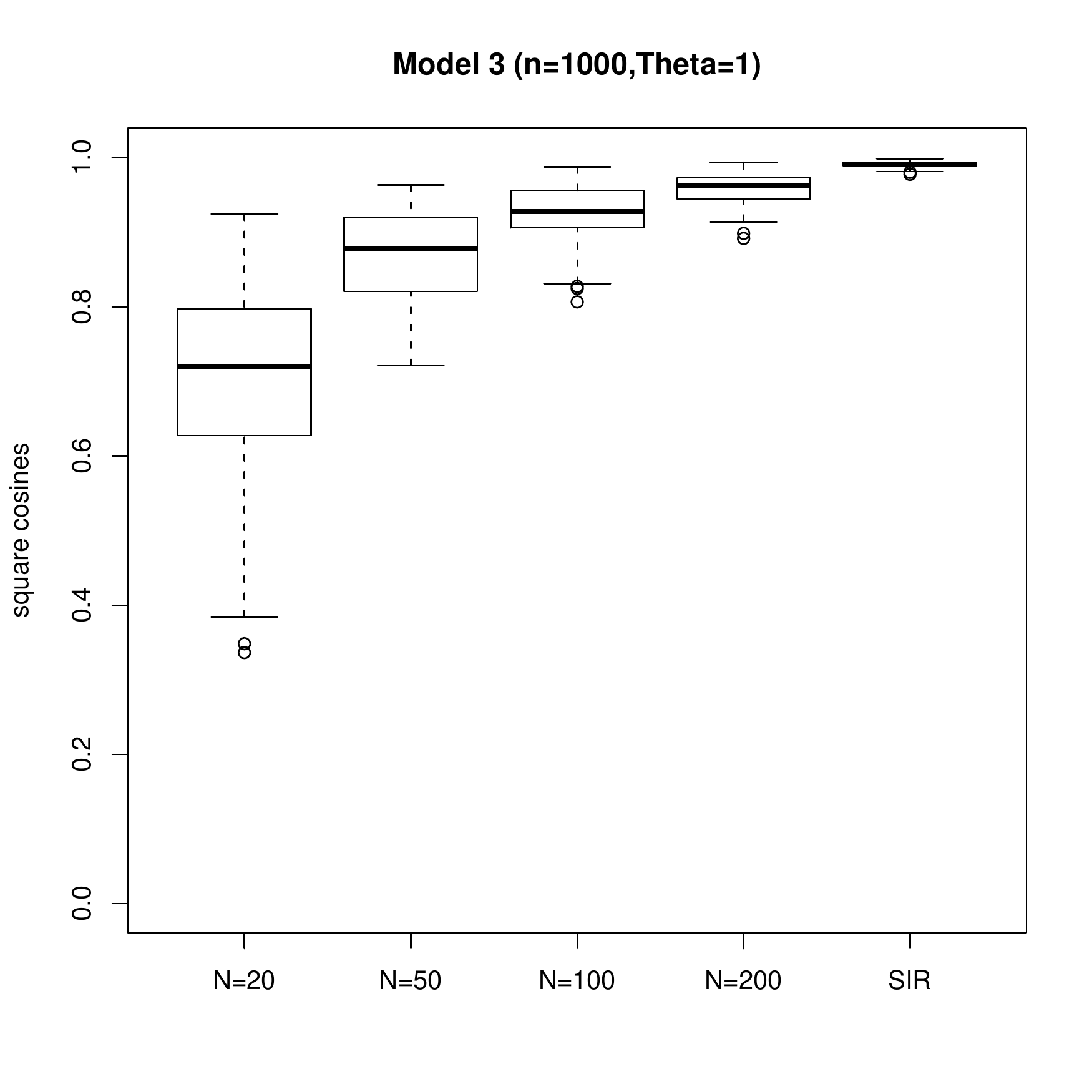} \\
   \includegraphics[width=7cm, height=6cm]{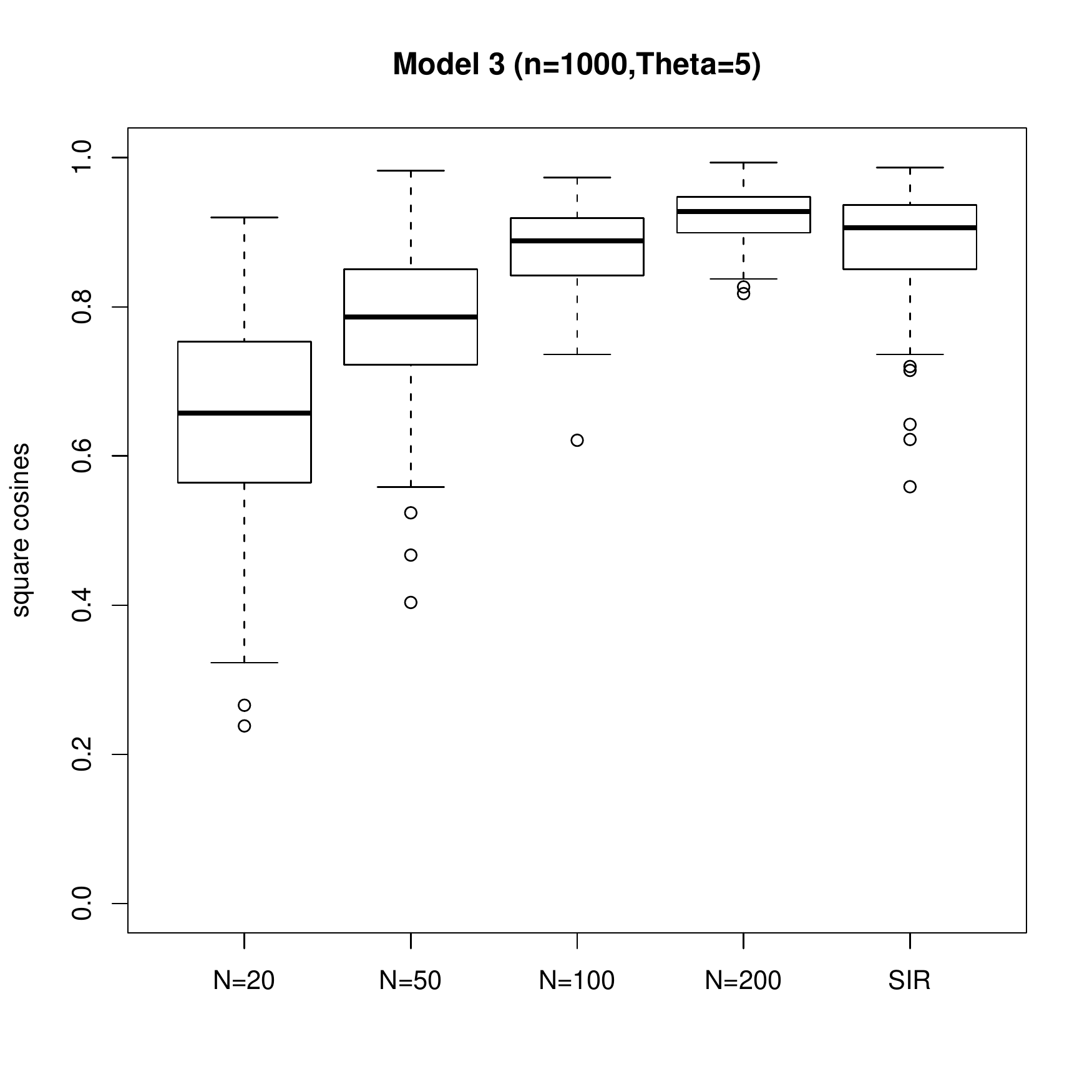} &
   \includegraphics[width=7cm, height=6cm]{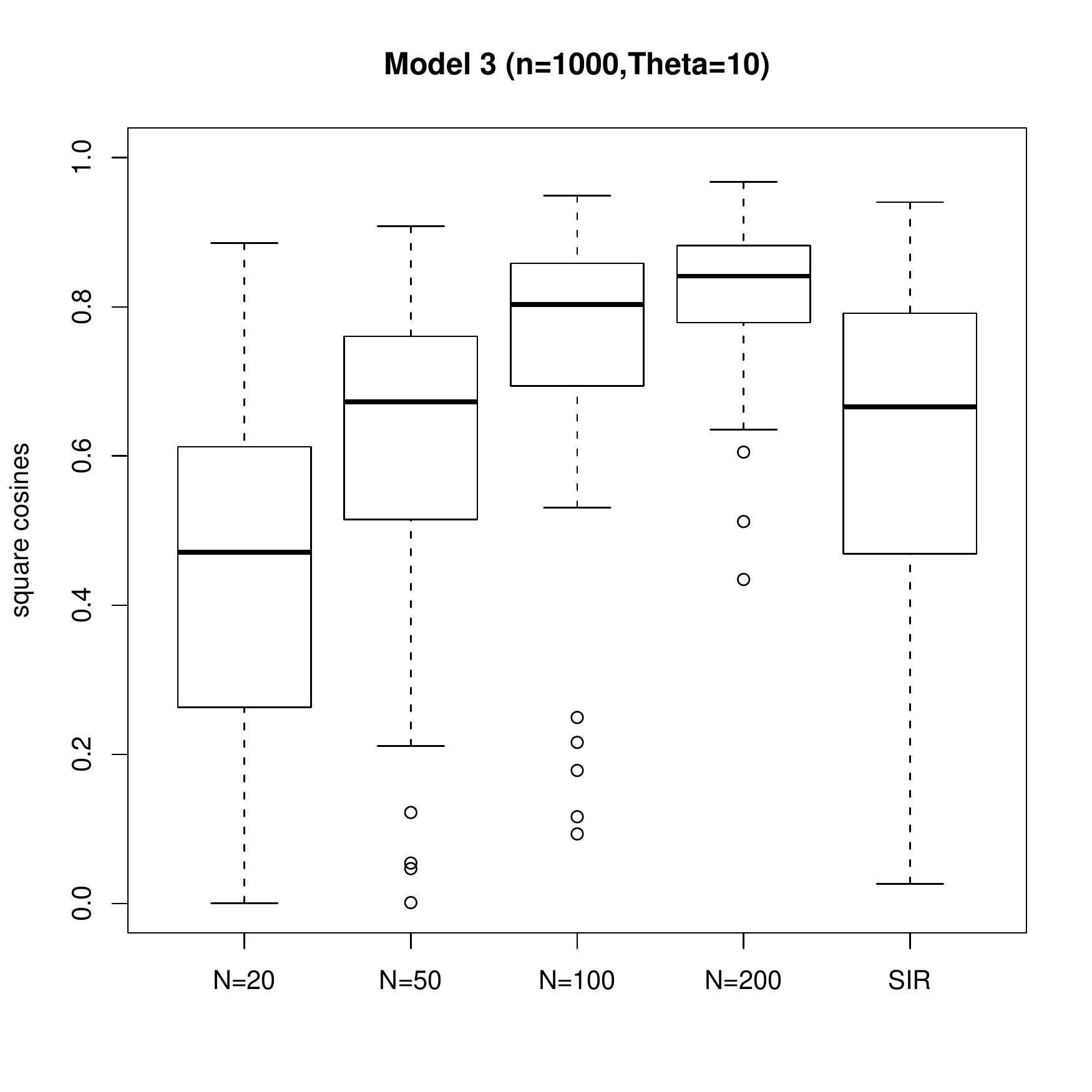}
\end{tabular}
\caption{Boxplots of the square cosines in the estimation of the regression parameter $\beta$ with $n=1000$}\label{figure2}
\end{figure}
\end{landscape}

\begin{table}[!htbp]
\centering
\begin{tabular}{c| c | c  | c | c  | c }
 \multicolumn{6}{c}{\textbf{Conditional expectation}}	\\
~			& Model $({\mathcal{M}}_1)$	&  Model $({\mathcal{M}}_2)$	&	Model $({\mathcal{M}}_3)$ 		& Model $({\mathcal{M}}_3)$ 		& Model $({\mathcal{M}}_3)$	\\
~			&	~	& ~		&	with $\theta=1$		&	with $\theta=5$	& with $\theta=10$		\\ \hline
true value		& 1		&	1  	&	2.72			&	1.22		&	1.11		\\
estimation		& 1.07		&	0.84	&	2.56			& 1.02			& 1.11			\\
relative error		& 0.07		&	-0.16	&	-0.06			&	-0.16		&	0.00	\\ 
\multicolumn{6}{c}{~} \\
\multicolumn{6}{c}{\textbf{Conditional variance}}	\\
~			& Model $({\mathcal{M}}_1)$	&  Model $({\mathcal{M}}_2)$	&	Model $({\mathcal{M}}_3)$ 		& Model $({\mathcal{M}}_3)$ 		& Model $({\mathcal{M}}_3)$	\\
~			&	~	& ~		&	with $\theta=1$		&	with $\theta=5$	& with $\theta=10$		\\ \hline
true value		& 1		&	1  	&	1			&	1		&	1		\\
estimation		& 1.01		&	0.92	&	1.14			& 1.05			& 	1.07	\\
relative error		& 0.01		&	-0.08	&	0.14			&	0.05		&	0.07		\\
\end{tabular}
\caption{Forecasting of $Y$ given $X=(0.5 , -0.5,1,0)'$} \label{tab:x1}
\end{table}

\vspace{3cm}

\begin{table}[!htbp]
\centering
\begin{tabular}{c| c | c  | c | c  | c }
 \multicolumn{6}{c}{\textbf{Conditional expectation}}	\\
~			& Model $({\mathcal{M}}_1)$	&  Model $({\mathcal{M}}_2)$	&	Model $({\mathcal{M}}_3)$ 		& Model $({\mathcal{M}}_3)$ 		& Model $({\mathcal{M}}_3)$	\\
~			&	~	& ~		&	with $\theta=1$		&	with $\theta=5$	& with $\theta=10$		\\ \hline
true value		& -0.56		&	-0.58	&	0.30			&	0.59		&0.64		\\
estimation		& -0.64		&	-0.51	&	0.52			& 0.80			& 0.61		\\
relative error		& 0.11		&	-0.12	&	0.71			&	0.36		&	-0.05		\\ 
\multicolumn{6}{c}{~} \\
\multicolumn{6}{c}{\textbf{Conditional variance}}	\\
~			& Model $({\mathcal{M}}_1)$	&  Model $({\mathcal{M}}_2)$	&	Model $({\mathcal{M}}_3)$ 		& Model $({\mathcal{M}}_3)$ 		& Model $({\mathcal{M}}_3)$	\\
~			&	~	& ~		&	with $\theta=1$		&	with $\theta=5$	& with $\theta=10$		\\ \hline
true value		& 1		&	0.69  	&	1			&	1		&	1		\\
estimation		& 1.01		&	0.80	&	0.77			& 1.01			& 	1.09		\\
relative error		& 0.01		&	0.15	&	-0.23			&	0.01		&	0.09		\\
\end{tabular}
\caption{Forecasting of $Y$ given $X=(-1/3 , 0.5,1,1)'$} \label{tab:x2}
\end{table}

\begin{figure}[!htbp]\label{figure3}
\centering
   \includegraphics[height=6cm,width=10cm]{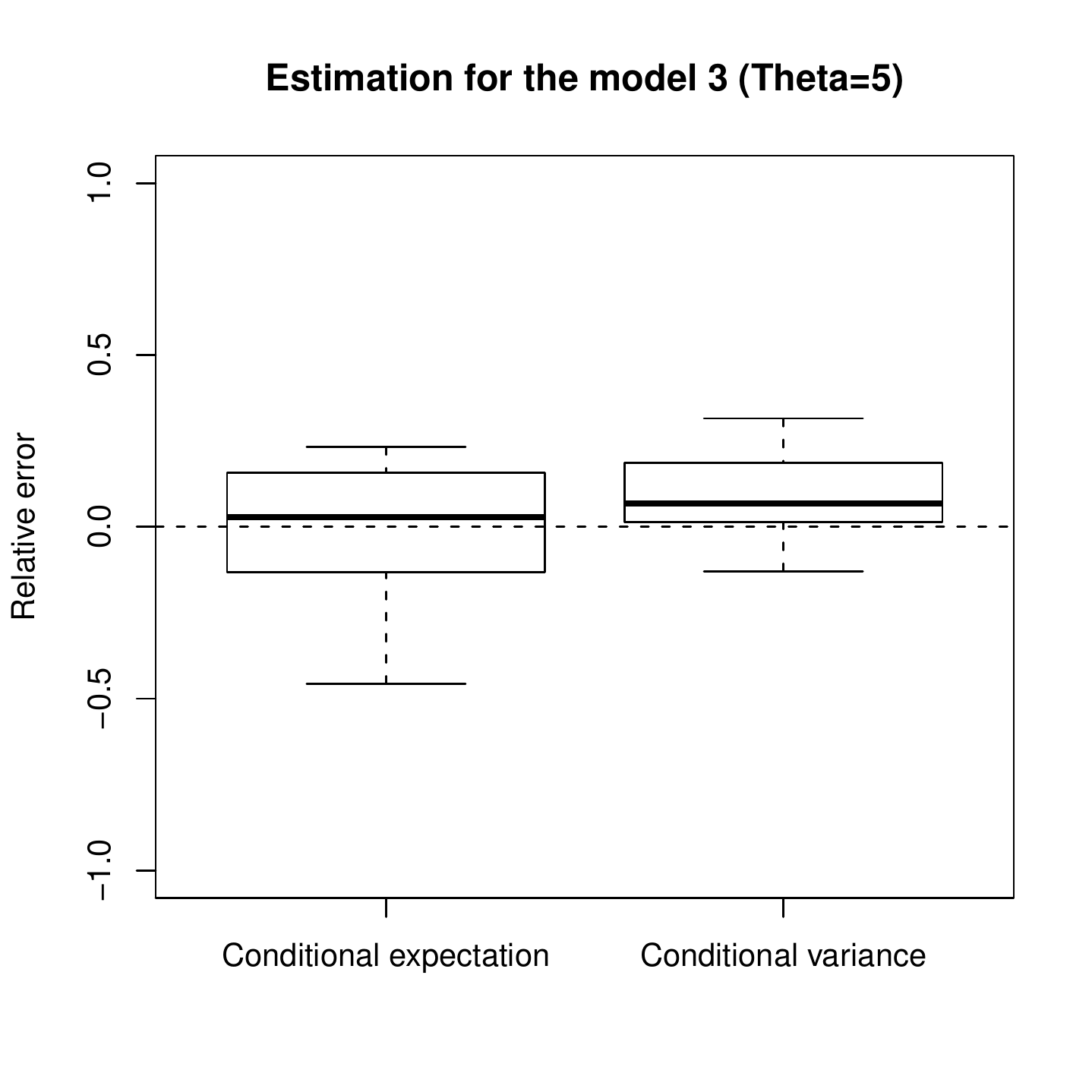} 
\caption{Relative error in the estimation of the conditional expectation and variance} \label{tab:boxpred}
\end{figure}

\newpage

\appendix
\section{Proofs}
\label{sproofs}

\subsection{Proof of Theorem \ref{label-pred1}}

Let us define $F(u)=\esp\big[ \phi(Y) | U=u \big]$ and $\hat{F}(\hat{u})=\esp\big[ \phi(\hat{Y}) | \hat{U} = \hat{u} \big]$.
By triangle inequality, we have
\begin{eqnarray*}
 \big\| F(U) - \hat{F}(\hat{U}) \big\|_p &\leq& \big\|F(U)-F(\hat{U}) \big\|_p 
~+~ \big\| F( \hat{U}) - \esp\big[ \phi(Y) | \hat{U} \big] \big\|_p  \\~&~& \qquad \qquad + ~\big\| \esp\big[ \phi(Y) | \hat{U} \big] - \hat{F}(\hat{U}) \big\|_p.
\end{eqnarray*}
Furthermore, using Lipschitz property of $\phi$ and $\tilde{f}$ and the independance between $U$ and $\epsilon$, we have
\begin{eqnarray*}
 \forall u, v \in \mathbf{R}^d, ~~ |F(u)-F(v)| &\leq & [\phi]_{Lip} [\tilde{f}]_{Lip} |u-v|.
\end{eqnarray*}
Then by $\mathbf{L}^p$-contraction property of conditional expectation, we get
$$ \big\| F( \hat{U}) - \esp\big[ \phi(Y) | \hat{U} \big] \big\|_p \leq  \big\|F(U)-F(\hat{U}) \big\|_p 
~\leq  [\phi]_{Lip} [\tilde{f}]_{Lip} \big\|U-\hat{U} \big\|_p.$$
Moreover since $\hat{F}( \hat{U}) - \esp\big[ \phi(Y) | \hat{U} \big] = \esp\big[ \phi(\hat{Y}) - \phi(Y) | \hat{U} \big]$, thus 
using again $\mathbf{L}^p$-contraction property of conditional expectation and Lipschitz property of $\phi$, we obtain:
$$\big\| \hat{F}( \hat{U}) - \esp\big[ \phi(Y) | \hat{U} \big] \big\|_p \leq  \big\| \phi(Y) - \phi( \hat{Y}) \big\|_p 
~\leq  [\phi]_{Lip}  \big\|Y-\hat{Y} \big\|_p.$$
This yields the expected inequality. $\Box$

\subsection{Proof of Theorem \ref{label-estimation1}}

Before proving Theorem~ \ref{label-estimation1}, we first give in Lemma~\ref{lemma1} a useful bound of the distance between $\hat{\Gamma}$ and $\hat{\Gamma}_N$.
\begin{lemme}\label{lemma1}
Under $(\mathcal{A}_{5 \to 10})$, we have
$ \| \hat{\Gamma} - \hat{\Gamma}_N \|_{\infty} \leq 2 d \| X - \hat{X}^N \|_p \| X \|_q.$
\end{lemme}

\begin{proof}
Let $A_N = \hat{\Gamma} - \hat{\Gamma}_N$. We control all the terms of this symmetric matrix hereafter.
	\begin{itemize}
	 \item \textbf{Study of the diagonal terms.} For all $1\leq i \leq d$, we have
	\begin{eqnarray*}
	 A_N (i,i) &=& \var \big[ \esp[X|\hat{Y}]_i \big] - \var \big[ \esp[ \hat{X}^N | \hat{Y}]_i \big] \\
	~	&=&    \esp \big[ \esp[X|\hat{Y}]_i^2 \big] - \esp[X]_i^2 - \esp \big[ \esp[\hat{X}^N|\hat{Y}]_i^2 \big] + \esp[\hat{X}^N]_i^2.
	\end{eqnarray*}
	Using the stationary property (\ref{stationnarity}), we obtain
	\begin{eqnarray*}
 	A_N (i,i) &=& \esp \big[ \esp[X|\hat{Y}]_i^2 \big] - \esp \big[ \esp[\hat{X}^N|\hat{Y}]_i^2 \big] \\
		~ &=& \esp \Big[ \big( \esp[X|\hat{Y}]_i - \esp[\hat{X}^N | \hat{Y}]_i \big) \big( \esp[X|\hat{Y}]_i + \esp[\hat{X}^N | \hat{Y}]_i \big) \Big].
	\end{eqnarray*}
	By H{\"o}lder inequality, $\mathbf{L}^p$ and $\mathbf{L}^q$-contraction property of conditional expectation and stationary property (\ref{stationnarity}), we get
	\begin{eqnarray*}
	 |A_N (i,i)| &\leq& \big\| \esp[ X - \hat{X}^N | \hat{Y} ]_i \big\|_p \big\| \esp[ X + \hat{X}^N| \hat{Y}]_i \big\|_q \\
	~&\leq& \| (X - \hat{X}^N)_i \|_p \| (X+\hat{X}^N)_i \|_q \\
	~&\leq& \| X-\hat{X}^N \|_p \big( \|X\|_q + \|\hat{X}^N \|_q \big) \\
	~&\leq& 2 \| X -\hat{X}^N\|_p \|X\|_q.
	\end{eqnarray*}

	\item \textbf{Study of the non-diagonal terms.} For all $(i,j)\in \{1,\dots,d\}^2$ such that $i \neq j$, we have
	\begin{eqnarray*}
	 A_N (i,j) &=& \cov\Big[ \esp[X|\hat{Y}]_i , \esp[X|\hat{Y}]_j\Big] - \cov\Big[ \esp[\hat{X}^N | \hat{Y}]_i , \esp[\hat{X}^N | \hat{Y}]_j \Big] \\
	~&=& \esp\Big[ \esp[X|\hat{Y}]_i \esp[X|\hat{Y}]_j - \esp[\hat{X}^N | \hat{Y}]_i \esp[\hat{X}^N |\hat{Y}]_j \Big] \\
	~&=& \esp\Big[ \Big( \esp[X|\hat{Y}]_i - \esp[\hat{X}^N | \hat{Y}]_i\Big) \Big( \esp[X|\hat{Y}]_j + \esp[\hat{X}^N | \hat{Y}]_j\Big) \Big] \\
	~&~& ~~~+ \esp\big[  \esp[\hat{X}^N | \hat{Y}]_i \esp[X|\hat{Y}]_j \big] - \esp\big[  \esp[\hat{X}^N | \hat{Y}]_j \esp[X|\hat{Y}]_i \big].
	\end{eqnarray*}
	By symmetry, we also have
	\begin{eqnarray*}
	 A_N (i,j) &=& \esp\Big[ \Big( \esp[X|\hat{Y}]_j - \esp[\hat{X}^N | \hat{Y}]_j\Big) \Big( \esp[X|\hat{Y}]_i + \esp[\hat{X}^N | \hat{Y}]_i\Big) \Big] \\
	~&~& ~~~+ \esp\big[  \esp[\hat{X}^N | \hat{Y}]_j \esp[X|\hat{Y}]_i \big] - \esp\big[  \esp[\hat{X}^N | \hat{Y}]_i \esp[X|\hat{Y}]_j \big].
	\end{eqnarray*}
	Then summing the two latest formulae, we get
	\begin{eqnarray*}
	2 A_N (i,j) &=& \esp\Big[ \Big( \esp[X|\hat{Y}]_i - \esp[\hat{X}^N | \hat{Y}]_i\Big) \Big( \esp[X|\hat{Y}]_j + \esp[\hat{X}^N | \hat{Y}]_j\Big) \Big] \\
	~&~& ~~~ + \esp\Big[ \Big( \esp[X|\hat{Y}]_j - \esp[\hat{X}^N | \hat{Y}]_j\Big) \Big( \esp[X|\hat{Y}]_i + \esp[\hat{X}^N | \hat{Y}]_i\Big) \Big]
	\end{eqnarray*}
	As for diagonal terms, by H{\"o}lder inequality and $\mathbf{L}^p$ and $\mathbf{L}^q$-contraction property of conditional expectation, we obtain
	\begin{eqnarray*}
	 |A_N (i,j)| &\leq& 2 \| X - \hat{X}^N \|_p \|X\|_q.
	\end{eqnarray*}
	\end{itemize}
	Finally, we obtain
	$$ \| A_N \|_\infty = \max_{1\leq i \leq d} \sum_{j=1}^d |A_N (i,j)| \leq 2 d \| X - \hat{X}^N \|_p \|X\|_q,$$
	and the proof of the lemma is complete. $\Box$
\end{proof}

\medskip

\noindent
Lemma~\ref{lemma1} yields that the eigenvalues of $\Sigma^{-1} \hat{\Gamma}_N$ tend to the eigenvalues of $\Sigma^{-1} \hat{\Gamma}$ as $N$ goes to infinity. Particularly, if $\hat{\rho}$ and $\hat{\rho}_N$ denote the principal eigenvalue of $\Sigma^{-1} \hat{\Gamma}$ and $\Sigma^{-1}\hat{\Gamma}_N$, we have
	$$ \hat{\rho}_N \to \hat{\rho} >0~~\textrm{as $N \to \infty$}.$$
Consider $z_N$ the principal eigenvector of norm $1$ of $\Sigma^{-1} \hat{\Gamma}_N$ such that $z_N = \alpha_N \beta + z_N^{ \perp }$ where $z_N^\perp \perp \beta$ and $\alpha_N \geq 0$. Thus we have
	\begin{eqnarray*}
	 \hat{\rho}_N z_N &=& \Sigma^{-1} \hat{\Gamma}_N z_N \\
		~		&=& \Sigma^{-1} \Big\{ ( \hat{\Gamma}_N - \hat{\Gamma}) z_N + \hat{\Gamma} (\alpha_N \beta + z_N^\perp ) \Big\}.
	\end{eqnarray*}
From SIR theory, the rank of $\hat{\Gamma}$ is $1$ and $\beta$ is a principal eigenvector of $\Sigma^{-1} \hat{\Gamma}$. Thus the kernel of $\Sigma^{-1} \hat{\Gamma}$ is $(d-1)$-dimensional and contains $z_N^\perp$. 
Consequently, we get $
	 \hat{\rho}_N z_N  - \alpha_N \hat{\rho} \beta = \Sigma^{-1}  (\hat{\Gamma}_N - \hat{\Gamma}) z_N  $.
Now, let us show that the sequence $(\alpha_N)_{N\geq1}$ has a strictly positive limit. We have
	\begin{eqnarray*}
	 \big| \hat{\rho}_N z_N  - \alpha_N \hat{\rho} \beta \big| &\leq & \big\| \Sigma^{-1} \big\|_{\infty} \big\| \hat{\Gamma}_N - \hat{\Gamma} \big\|_\infty \to 0 ~~\textrm{as $N \to \infty.$}
	\end{eqnarray*}
	Since  $(\hat{\rho}_N)$ has a nonzero limit and $\hat{x}_N$ is norm $1$ for any $N$, we have
	$\displaystyle  \lim_{N \to \infty} \alpha_N  = \frac{1}{|\beta|}.$
	Let $N_1$ be an integer such that for all $N \geq N_1$, $\alpha_N >0$. Let $C = \min_{ N \geq N_1} \alpha_N$. Moreover for all $N \geq N_1$, let
	$\displaystyle \beta_N = \frac{z_N}{\alpha_N} \frac{\hat{\rho}_N}{\hat{\rho}}.$
	Thus for all $N \geq N_1$, we have
	\begin{eqnarray*}
	\big| \beta_N  - \beta \big| &\leq & \frac{1}{C} \big\| \Sigma^{-1} \big\|_{\infty} \big\| \hat{\Gamma}_N - \hat{\Gamma} \big\|_\infty.
	\end{eqnarray*}
	Thanks to Lemma~\ref{lemma1} and Theorem~\ref{label-pierce}, we obtain the expected result. $\Box$

\subsection{Proof of Theorem \ref{th-cosca}}

\noindent
Let us consider the sequence $(\beta_N)$ defined in the previous proof. Then for all $N \geq N_1$, there exists some $\lambda_N$ such that $\tilde{\beta}_N = \lambda_N \beta_N$ and
\begin{eqnarray*}
 \cos^2 ( \tilde{\beta}_N , \beta) &=&  \cos^2 ({\beta}_N , \beta)  \\
 ~&=& \Big(  \frac{ \hat{\rho}_N }{ \hat{ \rho} } \Big)^2   \frac{  | \beta |^2  }{ | \beta_N|^2 } \to 1 ~\textrm{ as $N \to \infty$.} ~\Box
 \end{eqnarray*}

\subsection{Proof of Theorem \ref{label-forecast1}}

\noindent
Let $Y_N = f( \beta'_N X , \epsilon)$ and let $\widehat{Y_N}^m$ be its projection on an optimal $m$-grid in $\mathbf{L}^p$-norm. First, we show that the forecast error is smaller than a sum of four terms given in Lemma~\ref{label-lemma1-forecast1}. Then we control each term to complete the proof.

\begin{lemme} 
\label{label-lemma1-forecast1}
	For any Lipschitz function $\phi$, we have,  for all $N$ and $m$,
	\begin{eqnarray*}
	\Big\| \esp\big[ \phi(Y) | \beta'X \big] - \esp\Big[ \phi( \hat{Y}^m ) | \widehat{ \beta'_N X}^m \Big] \Big\|_1
	& \leq & 4 [f]_{Lip} [\phi]_{Lip} \|X\|_p | \beta - \beta_N | \\
	~&+&  2 [\phi]_{Lip} \Big\| Y_N - \widehat{Y_N}^m \Big\|_p \\
	~&+& 2 [f]_{Lip} [\phi]_{Lip} \Big\| \beta_N' X - \widehat{\beta_N' X}^m \Big\|_p \\
	~&+& [\phi]_{Lip} \|Y - \hat{Y}^m\|_p.
	\end{eqnarray*}
\end{lemme}

	\begin{proof}
	Let $V=\beta'X$, $V_N=\beta_N'X$ and $F(v)= \esp\big[ \phi(Y) | V=v \big]$. Thus, forecast error can be written this way:
	\begin{eqnarray*}
	\esp[ \phi(Y) | \beta'X] - \esp\Big[ \phi( \hat{Y}^m ) | \widehat{ \beta'_N X}^m \Big] &=& F(V) - F(V_N)\\
	~&+& F(V_N) - \esp\big[ \phi(Y) | \beta_N'X\big] \\
	~&+& \esp\big[ \phi(Y) | \beta_N'X\big] - \esp\big[ \phi(Y_N) | \beta_N'X\big] \\
	~&+& \esp\big[ \phi(Y_N) | \beta_N'X\big] - \esp\Big[ \phi\big(\widehat{Y_N}^m \big) | \widehat{\beta_N'X}^m \Big] \\
	~&+& \esp\Big[ \phi\big(\widehat{Y_N}^m \big) | \widehat{\beta_N'X}^m \Big] - \esp\Big[ \phi({Y_N}) | \widehat{\beta_N'X}^m \Big] \\
	~&+&  \esp\Big[ \phi({Y_N}) | \widehat{\beta_N'X}^m \Big]  - \esp\Big[ \phi( \hat{Y}^m ) | \widehat{ \beta'_N X}^m \Big].
	\end{eqnarray*}
	We will provide a majorant for each term in $\mathbf{L}^p$-norm.\\
	For the first one, we get
	\begin{eqnarray*}
	 \| F(V) - F(V_N) \|_p &\leq & [\phi]_{Lip} [f]_{Lip} \|V-V_N\|_p \\
	~&\leq & [\phi]_{Lip} [f]_{Lip} \| X \|_p |\beta - \beta_N|.
	\end{eqnarray*}
	For the second term, since $\sigma( \beta_N'X) \subset \sigma(X)$, we have
	\begin{eqnarray*}
	 F(V_N) - \esp\big[ \phi(Y) | \beta_N'X\big] &=& \esp\Big[ F(V_N)- \esp[\phi(Y) | X] \big| \beta'_N X \Big].
	\end{eqnarray*}
	Since $\esp[\phi(Y) | X] = \esp[\phi(Y) | \beta'X] = F(V)$ from model assumption, we obtain by $\mathbf{L}^p$-contraction property of conditional expectation:
	\begin{eqnarray*}
	 \big\| F(V_N) - \esp\big[ \phi(Y) | \beta_N'X\big] \big\|_p &\leq& \|F(V_N) - F(V)\|_p \\
	~&\leq& [\phi]_{Lip} [f]_{Lip} \| X\|_p |\beta-\beta_N|.
	\end{eqnarray*}
	For the third term, we get
	\begin{eqnarray*}
	 \Big\| \esp\big[ \phi(Y) | \beta_N'X\big] - \esp\big[ \phi(Y_N) | \beta_N'X\big] \Big\|_p &\leq&
	\| \phi(Y) - \phi(Y_N) \|_p \\
	~&\leq& [\phi]_{Lip} [f]_{Lip} \|X\|_p |\beta-\beta_N|.
	\end{eqnarray*}
	For the fourth term, we apply Theorem~\ref{label-pred1} to $Y_N = f(\beta'_N X, \epsilon)$ and we obtain:
	\begin{eqnarray*}
	 \Big\| \esp\big[ \phi(Y_N) | \beta_N'X\big] - \esp\Big[ \phi\big(\widehat{Y_N}^m \big) | \widehat{\beta_N'X}^m \Big] \Big\|_p &\leq&
	2 [\phi]_{Lip} [f]_{Lip} \big\| \beta'_N X - \widehat{\beta'_N X}^m \big\|_p \\~&+& [\phi]_{Lip} \big\|Y_N - \widehat{Y_N}^m \big\|_p.
	\end{eqnarray*}
	For the fifth term, we have
	\begin{eqnarray*}
	\Big\|  \esp\Big[ \phi\big(\widehat{Y_N}^m \big) | \widehat{\beta_N'X}^m \Big] - \esp\Big[ \phi({Y_N}) | \widehat{\beta_N'X}^m \Big]      \Big\|_p &\leq& \big\| \phi(Y_N) - \phi\big( \widehat{Y_N}^m \big) \big\|_p \\ 
	~&\leq& [\phi]_{Lip} \big\| Y_N - \widehat{Y_N}^m \big\|_p.
	\end{eqnarray*}
	Finally for the last term, we obtain
	\begin{eqnarray*}
	\Big\| \esp\Big[ \phi({Y_N}) | \widehat{\beta_N'X}^m \Big]  &-&\esp\Big[ \phi( \hat{Y}^m ) | \widehat{ \beta'_N X}^m \Big] \Big\|_p  \\
	~& \leq&  \| \phi(Y_N) - \phi(Y) \|_p + \|\phi(Y) - \phi\big( \hat{Y}^m\big) \|_p   \\
	~&\leq&  [\phi]_{Lip} \| Y_N - Y\|_p + [\phi]_{Lip}\|Y-\hat{Y}^m\|_p  \\
	~& \leq& [\phi]_{Lip} [f]_{Lip} \|X\|_p |\beta-\beta_N| + [\phi]_{Lip} \|Y - \hat{Y}^m\|_p. 
	\end{eqnarray*}
	Summing these six inequalities yields the expected result. $\Box$
	\end{proof}

\medskip

\noindent
The rest of the proof of Theorem~\ref{label-forecast1} splits into four parts.
	\begin{enumerate}[(i)]
	 
	\item Using (\ref{etoile}) and Theorem~\ref{label-pierce}, since $\|X\|_{p+\delta} < \infty$ there exist $D_1,D_2,D_3$ such that $\forall N \geq \max(C_0,D_3)$, we have
	\begin{equation}
	| \beta - \beta_N | \leq \frac{\Delta_0}{N^{1/d}}  \Big\{ D_1 \|X\|_{p+\delta}^{p+\delta} + D_2 \Big\}^{1/p}.
	\label{in1}
	\end{equation}

	\item Using Theorem~\ref{label-pierce} again, there exist $C_1,C_2,C_3$ such that $\forall m \geq C_3$, we have
	$$ \Big\| Y_N - \widehat{Y_N}^m \Big\|_p \leq \frac{1}{m} \Big\{ C_1 \|Y_N\|_{p+\delta}^{p+\delta} + C_2\Big\}^{1/p}. $$
	By triangle inequality, we get
	\begin{eqnarray*}
	 \|Y_N\|_{p+\delta} &\leq& \| Y \|_{p+\delta} + \| Y - Y_N \|_{p+\delta} \\
	~&\leq& \|Y\|_{p+\delta} + [f]_{Lip} \|X\|_{{p+\delta}} |\beta-\beta_N| \\
	~&\leq& \|Y\|_{p+\delta} + [f]_{Lip} \|X\|_{p+\delta} \Delta_0 \|X-\hat{X}^N\|_p~~\textrm{for $N\geq C_0$}.
	\end{eqnarray*}
	From (\ref{aaa}), for all $N \geq \max(C_0,D_3)$ we have
	$$
	\|Y_N\|_{p+\delta}^{p+\delta} \leq \Big\{ \|Y\|_{p+\delta} + \frac{\Delta_1 \Delta_2}{N^{1/d}}\Big\}^{p+\delta}
	$$
	where $
	\Delta_1 = \Delta_0 [f]_{Lip} \|X\|_{p+\delta} $  and $\Delta_2 = \Big( D_1 \|X\|_{p+\delta}^{p+\delta} + D_2 \Big)^{1/p}$.
	Finally  $\forall m \geq C_3$ and $\forall N \geq \max(C_0,D_3)$, we have
		\begin{equation} \label{in2}
		 \Big\| Y_N - \widehat{Y_N}^m \Big\|_p \leq \frac{1}{m} \Big\{ C_1 \Big( \frac{\Delta_1 \Delta_2}{N^{1/d}} + \|Y\|_{p+\delta}\Big)^{p+\delta} + C_2 \Big\}^{1/p}.
		  \end{equation}
	
	\item Using Theorem~\ref{label-pierce} again, there exist $ C'_1, C'_2, C'_3 > 0$ such that $\forall m \geq C'_3$,
	\begin{eqnarray*}~
	 \Big\| \beta_N' X - \widehat{\beta_N' X}^m \Big\|_p &\leq& \frac{1}{m} \Big\{ C'_1 \|\beta'_N X\|_{p+\delta}^{p+\delta} + C'_2\Big\}^{1/p} \\
	~&\leq& \frac{1}{m} \Big\{ C'_1 |\beta_N|^{p+\delta} \|X\|_{p+\delta}^{p+\delta} + C'_2\Big\}^{1/p}.
	\end{eqnarray*}
	By (\ref{etoile}) we get
	$$ \forall N \geq C_0,~ \big|\beta_N\big| \leq \big|\beta\big| + \Delta_0 \big\| X- \hat{X}^N\big\|_p. $$
	Using (\ref{aaa}), we have  $\forall N\geq \max(C_0,D_3)$,
	\begin{equation}\label{in3}
	 \Big\| \beta_N' X - \widehat{\beta_N' X}^m \Big\|_p
	 \leq \frac{1}{m} \Big\{ C'_1 \|X\|_{p+\delta}^{p+\delta} \Big( \frac{\Delta_0 \Delta_2}{N^{1/d}} + |\beta|  \Big)^{p+\delta} + C'_2 \Big\}^{1/p}.  
	\end{equation}

	\item Using Theorem~\ref{label-pierce} again, there exist $C''_1,C_2'',C_3''$ such that thanks to $Y \in \mathbf{L}^{p+\delta}$, for all $m \geq C_3''$, we have
	\begin{equation}\label{in4}
	 \|Y- \hat{Y}^m \|_p \leq \frac{1}{m} \big\{ C_1'' \|Y\|_{p+\delta}^{p+\delta} + C_2''\big\}^{1/p}. 
	\end{equation}	
	\end{enumerate}

\medskip

\noindent
Plugging (\ref{in1}), (\ref{in2}), (\ref{in3}) and (\ref{in4}) in the inequality of Lemma~\ref{label-lemma1-forecast1}, we have $\forall m \geq \max(C_3,C'_3,C_3'')$ and $\forall N \geq \max(C_0 , D_3)$,
	\begin{eqnarray*}
	\Big\| \esp\big[ \phi(Y) | \beta'X \big] &- &\esp\Big[ \phi  ( \hat{Y}^m  ) | \widehat{ \beta'_N X}^m \Big] \Big\|_1  \\
	~&     \leq& \frac{1}{{N^{1/d}}} 4 [f]_{Lip} [\phi]_{Lip} \|X\|_p \Delta_0 \Big\{ D_1 \|X\|_{p+\delta}^{p+\delta} + D_2 \Big\}^{1/p}\\
	~&+&   \frac{1}{{m}} 2 [\phi]_{Lip} \Big\{ C_1 \Big( \frac{\Delta_1 \Delta_2}{{N^{1/d}}} + \|Y\|_{p+\delta} \Big)^{p+\delta} + C_2 \Big\}^{1/p} \\
	~&+&   \frac{1}{{m}} 2 [f]_{Lip} [\phi]_{Lip} \Big\{C_1' \|X\|_{p+\delta}^{p+\delta} \Big( \frac{\Delta_0 \Delta_2}{{N^{1/d}}} + |\beta|\Big)^{p+\delta} + C'_2 \Big\}^{1/p} \\
	~&+&   \frac{1}{{m}} \Big\{ C_1'' \|Y\|_{p+\delta}^{p+\delta} + C_2''\Big\}^{1/p}.
	\end{eqnarray*}
	This achieves the proof. $\Box$

\renewcommand{\refname}{Bibliography}
\bibliographystyle{plain}
\bibliography{SIR}
\nocite{*}

\end{document}